\documentclass{article}
\usepackage{amssymb,amsmath}
\usepackage{graphicx}

\def\qed{\hfill {\hbox{${\vcenter{\vbox{               %HOLLOW SQUARE
   \hrule height 0.4pt\hbox{\vrule width 0.4pt height 6pt
   \kern5pt\vrule width 0.4pt}\hrule height 0.4pt}}}$}}}

\def\tr{\triangleright}

\newtheorem{theorem}{Theorem}
\newtheorem{definition}{Definition}

\newtheorem{proposition}[theorem]{Proposition}

\newtheorem{example}{Example}
\newtheorem{remark}[example]{Remark}

\newenvironment{proof}[1][Proof]{\smallskip\noindent{\bf #1.}\quad}%
{\qed\par\medskip}

\date{}

\title{\Large \textbf{The column group and its link invariants}}

\author{Johanna Hennig \and Sam Nelson} 

\begin{document}

%\baselineskip=2\baselineskip

\maketitle

\begin{abstract}
The \textit{column group} is a subgroup of the symmetric group on 
the elements of a finite blackboard birack generated by the column
permutations in the birack matrix. We use subgroups of the column 
group associated to birack homomorphisms to define an enhancement
of the integral birack counting invariant and give examples which 
show that the enhanced invariant is stronger than the unenhanced 
invariant.
\end{abstract}

\parbox{6in}{\textsc{Keywords:} blackboard biracks, link invariants, 
enhancements of counting invariants

\smallskip

 }

\section{\large \textbf{Introduction}}

Introduced in \cite{FRS}, a \textit{birack} is a solution to the 
set-theoretic Yang-Baxter equation satisfying certain invertibility
conditions. A \textit{blackboard birack} is a type of strong birack 
with axioms corresponding to oriented blackboard-framed Reidemeister moves 
\cite{bbr}. Special cases include quandles \cite{J,M}, racks \cite{FR}, 
strong biquandles \cite{FJK} and semiquandles \cite{HN}.

In \cite{bbr} an invariant of knots and links is defined from any
finite blackboard birack $X$ by counting labelings of a knot or link
diagram with elements of $X$ over a set of framings of the knot or 
link determined by a quantity known as the \textit{birack rank} of $X$.

In this paper we describe an enhancement of the birack counting invariant
defined using a group determined by the columns of the birack operation
matrices known as the \textit{column group}. The paper is organized as
follows. In section \ref{B} we review the basics of blackboard biracks.
In section \ref{CG} we define the column group and make a few useful 
observations. In section \ref{E} we use the column group to enhance 
the blackboard birack counting invariants and give some examples and
computations. In section \ref{Q} we collect questions for future 
research.

\section{\large \textbf{Blackboard biracks}}\label{B}

In this section we review the basics of blackboard biracks. See \cite{bbr}
for more.

A \textit{blackboard framed link} is an equivalence class of link
diagrams under the \textit{blackboard framed Reidemeister moves}:
\[\scalebox{0.93}{$\begin{array}{ccc}
\includegraphics{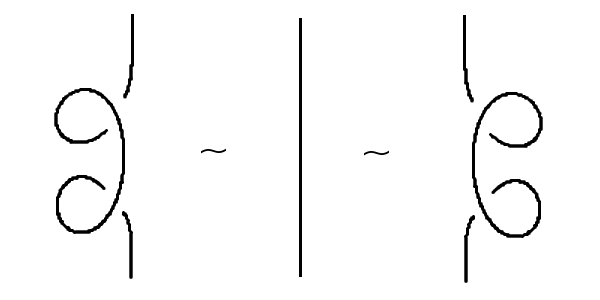} &
\includegraphics{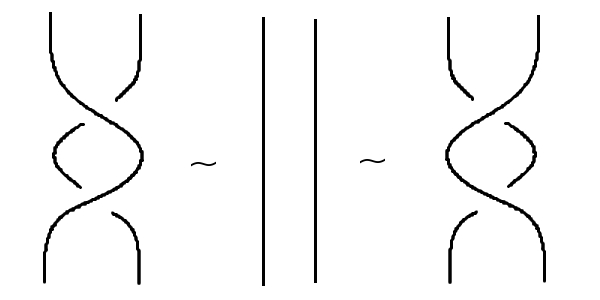} &
\includegraphics{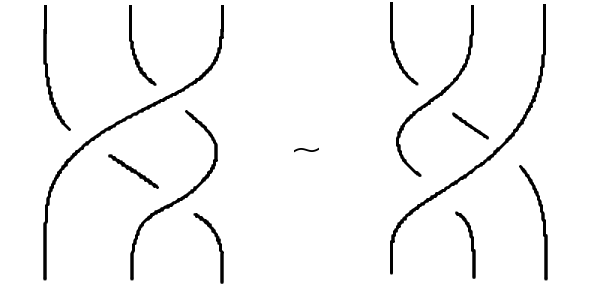} \\
\mathrm{Type\ I} &
\mathrm{Type\ II} &
\mathrm{Type\ III} \\
\end{array}$}
\]
Including a choice of orientation for each component gives us 
\textit{blackboard framed oriented links}. Each component $C_k$ has a
\textit{writhe} $w_k$ equal to the sum of the crossing signs over the 
set of crossings where both strands belong to $C_k$:
\[\includegraphics{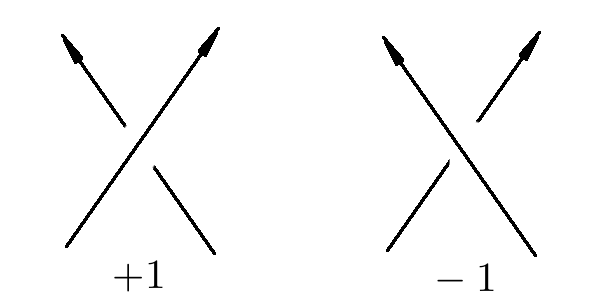}\]
The writhe of each component is an invariant of oriented blackboard framed 
isotopy. Assuming a fixed choice of ordering of the components of a link 
$L=C_1\cup\dots\cup C_c$, we obtain a \textit{writhe vector} 
$\mathbf{w}=(w_1,\dots,w_c)\in \mathbb{Z}^c$. Setting these writhe vectors
equal to to each other yields \textit{ambient isotopy} of unframed knots 
and links; this operation is equivalent to replacing the blackboard framed
type I moves with the usual single-kink versions.

\begin{definition}\textup{
Let $X$ be a set. A map $B=(B_1(x,y),B_2(x,y)):X\times X\to X\times X$ is 
\textit{strongly invertible} if it satisfies the following conditions:
\begin{list}{}{}
\item[(i)]{$B$ is \textit{invertible}, i.e. there exists a map 
$B^{-1}:X\times X \to X\times X$ satisfying \[BB^{-1}=\mathrm{Id}=B^{-1}B,\]}
\item[(ii)]{$B$ is \textit{sideways invertible}, i.e. there exists a unique
invertible map $S:X\times X \to X\times X$ satisfying for all $x,y\in X$
\[S(B_1(x,y),x)=(B_2(x,y),y),\]
and}
\item[(iii)]{$B$ is \textit{diagonally invertible}, i.e. the restrictions of 
the components of $S$ to the diagonal $\Delta=\{(x,x)\ |\ x\in X\}$
are bijections.}
\end{list}
}\end{definition}

\begin{definition}\textup{
Let $X$ be a set and $B:X\times X\to X\times X$ a 
strongly invertible map. The bijection $\pi:X\to X$ defined by
\[\pi(x)=S_2|_{\Delta}^{-1}\circ S_1|_{\Delta}\]
is called the \textit{kink map} of the pair $(X,B)$. The exponent of $\pi$,
i.e. the minimal integer $N\ge 1$ satisfying $\pi^N=\mathrm{Id}$ 
(or $\infty$ if no such $N$ exists), is called the \textit{birack rank} of 
\textit{birack characteristic} of $(X,B)$.
}\end{definition}

\begin{definition} \textup{
A \textit{blackboard birack} is a set $X$ with a strongly invertible map
$B:X\times X\to X\times X$ which satisfies the \textit{set-theoretic 
Yang-Baxter equation}:
\[(B\times \mathrm{Id})(\mathrm{Id}\times B)(B\times \mathrm{Id})
=(\mathrm{Id}\times B)(B\times \mathrm{Id})(\mathrm{Id}\times B)\]
where $\mathrm{Id}:X\to X$ is the identity map.}
\end{definition}

\begin{example}
\textup{Let $X$ be a module over the ring
$\tilde\Lambda=\mathbb{Z}[t^{\pm 1},s,r^{\pm 1}]/(s^2-(1-tr)s)$. Then it is 
easy to check (see \cite{bbr}) that 
\[B(x,y)=(ty+sx,rx)\]
defines a blackboard birack structure with $\pi(x)=(tr+s)$. We call this
a \textit{$(t,s,r)$-birack}.}
\end{example}

\begin{example}
\textup{Let $X$ be a set. Two bijections $\tau,\sigma:X\to X$ define a
blackboard birack structure on $X$ by 
\[B(x,y)=(\tau(y),\sigma(x))\]
if and only if $\tau\sigma=\sigma\tau$. We call this a \textit{constant 
action birack}; such a birack has kink map $\pi=\sigma\tau^{-1}$.}
\end{example}

\begin{example}
\textup{Many previously studied algebraic structures in knot theory are 
special cases of blackboard biracks:
\begin{itemize}
\item{A blackboard birack in which $\pi(x)=\mathrm{Id}$ is a 
\textit{strong biquandle} \cite{FJK}}
\item{A blackboard birack in which $B_2(x,y)=x$ for all $x,y\in X$ is a
\textit{rack} \cite{FR}}
\item{A blackboard birack in which $\pi(x)=\mathrm{Id}$ and $B_2(x,y)=x$ 
for all 
$x,y\in X$ is a \textit{quandle} \cite{J,M}.}
\end{itemize}
}\end{example}

Given a finite set $X=\{x_1,\dots, x_n\}$ we can define a blackboard
birack structure on $X$ by specifying the operation tables of the components
of $B$ with a matrix $M=[M_{B_1}|M_{B_2}]$ with two $n\times n$ blocks 
$M_{B_1}$ and $M_{B_2}$ whose $(i,j)$ entries respectively are $k$ and $l$ 
where $B_1(x_j,x_i)=x_k$ and $B_2(x_i,x_j)=x_l$.\footnote{Notice the reversed 
order of the inputs in $B_1$; this is for compatibility with notation in 
previous work.} Given such matrices, we check cases to determine
whether the blackboard birack axioms are satisfied (or rather, have a 
computer do so for us). It is not hard to see, for example, that such
matrices must have columns which are permutations.

\begin{example}
\textup{Let $X=\{1,2,3,4\}$ with $\tau=(12)$ and $\sigma=(34)$. Then 
$\tau\sigma=\sigma\tau$ so we have a constant action birack with nontrivial
kink map $\pi=(12)(34)$. The birack matrix is given by}
\[M=\left[\begin{array}{cccc|cccc}
2 & 2 & 2 & 2 & 1 & 1 & 1 & 1 \\
1 & 1 & 1 & 1 & 2 & 2 & 2 & 2 \\
3 & 3 & 3 & 3 & 4 & 4 & 4 & 4 \\
4 & 4 & 4 & 4 & 3 & 3 & 3 & 3 \\
\end{array}\right].\]
\end{example}

\begin{definition}
\textup{Let $(X,B)$ and $(X',B')$ be blackboard biracks.
As with other algebraic structures, we have the following useful notions:
\begin{itemize}
\item{A map $f:X\to X'$ satisfying $B(f(x),f(y))=(f(B_1(x,y)),f(B_2(x,y)))$ is
a \textit{homomorphism} of biracks,}
\item{A subset $Y\subset X$ such that the restriction $B_Y=B|_{Y\times Y}$ 
defines a blackboard birack is a \textit{subbirack} of $X$.}
\end{itemize}}
\end{definition}

A blackboard birack can be used to label the semiarcs in an oriented
blackboard-framed link diagram as indicated. Strong invertibility guarantees
that such labelings are preserved under blackboard framed Reidemeister moves
I and II in the sense that every labeling satisfying the pictured labeling 
condition at every crossing before a move corresponds to a unique such 
labeling after the move, while the Yang-Baxter condition does the 
same for type III moves. See \cite{bbr} for more.

\[\includegraphics{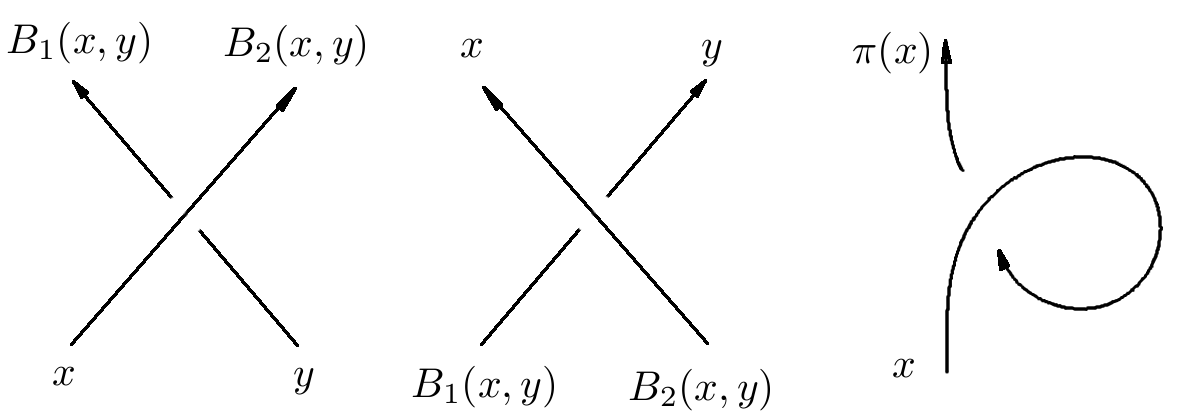}\]

If $(X,B)$ is a blackboard birack, then the set of labelings of a blackboard
framed oriented link diagram $L$ such that the crossing conditions pictured 
above are satisfied at every crossing is an invariant of blackboard-framed 
isotopy known as the \textit{basic counting invariant}, denoted 
$\mathrm{Hom}(BBR(L),(X,B))$. If $(X,B)$ has finite birack rank $N$, then
two blackboard-framed isotopic diagrams with congruent writhe vectors 
modulo $N$ are related by the blackboard-framed oriented Reidemeister
moves together with the \textit{$N$-phone cord move:}

\[\includegraphics{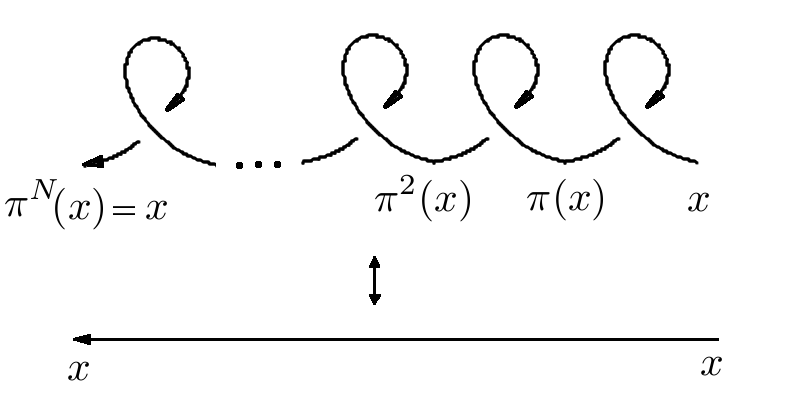}\]

Two such diagrams then have the same basic counting invariant with respect 
to $(X,B)$. Thus, the quantity
\[\Phi_{(X,B)}^{\mathbb{Z}}(K)=\sum_{\mathbf{w}\in(\mathbb{Z}_N)^c} 
|\mathrm{Hom}(BBR(K,\mathbf{w}),{X,B})|\]
is an invariant of ambient isotopy of knots and links, called the
\textit{integral blackboard birack counting invariant}.

An \textit{enhancement} of $\Phi^{\mathbb{Z}}_{(X,B)}(K)$ is a generally
stronger invariant which associates to each labeling a signature which 
is invariant under birack-labeled Reidemeister moves. One standard example
is the \textit{image enhanced blackboard birack counting invariant}
\[\Phi_{(X,B)}^{\mathrm{Im}}(K)=\sum_{\mathbf{w}\in(\mathbb{Z}_N)^c} 
\left(\sum_{f\in \mathrm{Hom}(BBR(K,\mathbf{w}),{X,B})} t^{|\mathrm{Im}(f)|}\right)\]
where $\mathrm{Im}(f)$ is the image of the labeling $f$ regarded as
a homomorphism from the fundamental blackboard birack of $K$ to $(X,B)$,
i.e. the smallest subbirack of $(X,B)$ containing all of the labels appearing 
in $f$. Another standard example is the \textit{writhe enhanced blackboard 
birack counting invariant} given by
\[\Phi_{(X,B)}^{W}(K)=\sum_{\mathbf{w}\in(\mathbb{Z}_N)^c} 
|\mathrm{Hom}(BBR(K,\mathbf{w}),{X,B})|q^{\mathbf{w}}\]
where $q^{(w_1,\dots,w_c)}=\prod_{k=1}^c q^{w_k}$. This enhancement keeps track of 
which writhe vectors contribute which colorings, and for certain racks 
determines the linking number mod $N$ for links with two components \cite{N3}.

Other examples of enhancements are known in special cases, such as 
quandle/biquandle/rack 2-cocycle enhancements \cite{CES, CJKLS, N3},
quandle/rack/biquandle polynomials \cite{N,N2, CN}, and various enhancements 
which use extra structure of the labeling objects, e.g. symplectic quandle
enhancements \cite{NN} and Coxeter rack enhancements \cite{NW}.

\section{\large \textbf{The column group}}\label{CG}

We can now define the column group of a blackboard birack.

Let $(X,B)$ be a blackboard birack with $n$ elements specified by a birack 
matrix $M$. As we have noted, the columns of $M$ determine permutations in 
the symmetric group $S_n$ -- for each $x_j\in X$, define 
$\tau_j,\sigma_j:\{1,2,\dots, n\}\to \{1,2,\dots, n\}$ by
$\tau_j(i)=k$ and $\sigma_j(i)=l$ where $x_k=B_1(x_j,x_i)$ and
$x_l=B_2(x_i,x_j)$. That is, $\tau_j$ and $\sigma_j$ are the permutations
determined by the $j$th column of $M_{B_1}$ and $M_{B_2}$ respectively.
We will call $\tau_j$ and $\sigma_j$ the \textit{upper} and \textit{lower 
column permutations} of the element $x_j$ respectively. Note that if $(X,B)$ 
is a rack then $\sigma_j=\mathrm{Id}$ for all $j$, and if $(X,B)$ is a 
quandle, then $\tau_j$ has $j$ as a fixed point for each $j$.

\begin{definition}
\textup{The \textit{column group} $CG(X)$ of a finite blackboard birack $X$ 
with $n$ elements is the subgroup of $S_n$ generated by the elements 
$\tau_j,\sigma_j\in S_n$ corresponding to the columns of the birack matrix 
$M_{(X,B)}$. More generally, if $S\subset X$ is a subbirack then the 
\textit{column subgroup}
$CG(S\subset X)$ is the subgroup of $CG(X)$ generated by the permutations 
corresponding to the columns of the elements of $S$.}
\end{definition}

\begin{example}
\textup{Consider the $(t,s,r)$-birack $X=\mathbb{Z}_3$ with $t=1$, $r=2$ and
$s=2$. $X$ has birack matrix below with the listed
upper and lower column permutations.}
\[\begin{array}{|c|ccc|ccc|} \hline
M_{(X,B)} & \tau_1 & \tau_2 & \tau_3 & \sigma_1 & \sigma_2 & \sigma_3 \\ \hline
& & & & & & \\
\left[\begin{array}{ccc|ccc}
1 & 3 & 2 & 1 & 1 & 1 \\
2 & 1 & 3 & 3 & 3 & 3 \\
3 & 2 & 1 & 2 & 2 & 2 \\
\end{array}\right] & () & (132) & (123) & (23) & (23) & (23) \\ 
& & & & & &  \\ \hline
\end{array}
\]
\textup{Thus, $CG(X)$ is the dihedral group of six elements; the subbirack
$S=\{1\}$ has column subgroup $CG(S\subset X)\cong \mathbb{Z}_2$.}
\end{example}

\begin{remark}\textup{If $(X,B)$ is quandle or rack, then the operation
$\tr$ defined by $x\tr y=B_1(y,x)$ is self-distributive. In 
this case, the column group $CG(X)$ is a subgroup of the automorphism group 
$\mathrm{Aut}(X)$ of $(X,B)$, sometimes called the \textit{inner automorphism 
group} of $X$. The column group is also related to the \textit{operator 
group} defined in \cite{FR}. In the more general setting of blackboard 
biracks, however, the columns need not be automorphisms, so for simplicity 
we prefer the term ``column group.''}
\end{remark}

\begin{proposition} 
Let $(X,B)$ and $(X',B')$ be finite blackboard biracks. If there 
exists an isomorphism of biracks $\phi:X\to X'$, then $CG(X)$ is isomorphic to 
$CG(X')$.
\end{proposition}

\begin{proof}
We will show that $CG(X)$ and $CG(X')$ have presentations which differ
only by relabeling. 

Let us denote $B_1(x,y)=y\tr_1 x$, $B_2(x,y)=x\tr_2 y$. Then $\phi$ a birack 
isomorphism says 
\[\phi(x\tr_1 y)=\phi(x)\tr_1 \phi(y)\quad \mathrm{and}\quad 
\phi(x\tr_2 y)=\phi(x)\tr_2 \phi(y).\]

Then if $X=\{x_1,\dots, x_n\}$ we have $X'=\{\phi(x_1),\dots, \phi(x_n)\}$. 
Let us abbreviate $\phi(x_i)$ as $\phi(i)$. Then the column groups $CG(X)$ 
and $CG(X')$ are generated by 
\[\{\tau_1,\dots,\tau_n,\sigma_1,\dots, \sigma_n\}\quad
\mathrm{and}
\quad \{\tau_{\phi(1)},\dots,\tau_{\phi(n)},\sigma_{\phi(1)},\dots,
\sigma_{\phi(n)}\}\] respectively. Note also that the finiteness of $X$
implies for any $\tau_i$ and $\sigma_i$ we have $\tau_i^{-1}=\tau_i^l$ and
$\sigma_i^{-1}=\sigma_i^{m}$ for some $l,m>0$.

Then
\[\phi(x_i \tr_1 x_j) = \phi(x_i)\tr_1 \phi(x_j) \quad \iff \quad 
\phi(\tau_j(i))=\tau_{\phi(j)}\left(\phi(i)\right)\] and
\[\phi(x_i\tr_2 x_j) = \phi(x_i)\tr_2(\phi(x_j) \quad \iff \quad 
\phi(\sigma_j(i))=\sigma_{\phi(j)}\left(\phi(i)\right)\] 

Let $\delta_i,\epsilon_i\in\{0,1\}$ and  for all
$i,j\in\{1,\dots,n\}$ set $x_i\tr_1^0x_j=x_i\tr_2^0x_j=x_i$. 
It follows that for any relation 
$\tau_{i_1}^{\delta_1}\sigma_{i_1}^{\epsilon_1}\circ \dots\circ \tau_{i_k}^{\delta_k}\sigma_{i_k}^{\epsilon_k}=\mathrm{Id}$ 
satisfied in $CG(X)$, we have for all $x_j\in X$
\begin{eqnarray*}
\phi(x_j) & = & \phi(\tau_{i_1}^{\delta_1}\sigma_{i_1}^{\epsilon_1} \circ 
\dots\circ \tau_{i_k}^{\delta_k}\sigma_{i_k}^{\epsilon_k}(x_j)) \\
& = & \phi(((\dots (x_j \tr_2^{\epsilon_k} x_{i_k})\tr_1^{\delta_k} x_{i_k})\dots\tr_2^{\epsilon_1} x_{i_1})\tr_1^{\delta_1} x_{i_1}) \\
& = & (\dots ((\phi(x_j) \tr_2^{\epsilon_k} \phi(x_{i_k}))\tr_1^{\delta_k} \phi(x_{i_k}))\dots \tr_2^{\epsilon_1} \phi(x_{i_1}))\tr_1^{\delta_1} \phi(x_{i_1}) \\
& = & \tau_{\phi(i_1)}^{\delta_1}\sigma_{\phi(i_1)}^{\epsilon_1} \circ \dots\circ 
\tau_{\phi(i_k)}^{\delta_k}\sigma_{\phi(i_k)}^{\epsilon_k}(\phi(x_j)) \\
\end{eqnarray*}
and the relation 
$\tau_{\phi(i_1)}^{\delta_1}\sigma_{\phi(i_1)}^{\epsilon_1}\circ \dots\circ
\tau_{\phi(i_k)}^{\delta_k}\sigma_{\phi(i_k)}^{\epsilon_k}=\mathrm{Id}$
is satisfied in $CG(X')$. 

Replacing $\phi$ with $\phi^{-1}$ shows that every 
relation satisfied in $CG(X')$ arises in this way. Thus, $CG(X)$ and $CG(X')$ 
have presentations which differ only by relabeling, and $CG(X)\cong CG(X')$.
\end{proof}

Note that, like quandle, biquandle and rack polynomials, the column 
subgroup of a subbirack carries information about how the subbirack is 
embedded in the overall birack. In particular, the column subgroup 
$CG(S\subset X)$ is not in general isomorphic to the column group $CG(S)$ 
considered as a stand-alone birack; isomorphic subbiracks $S\subset X$ and 
$T\subset X$ embedded differently in $X$ generally have non-isomorphic column 
groups $CG(S\subset X)\not\cong CG(T\subset X)$, as the next example 
illustrates.

\begin{example}
\textup{The two 2-element subbiracks $\{1,2\}$ and $\{3,4\}$ of the
birack with birack matrix}
\[M_{(X,B)}=\left[\begin{array}{cccc|cccc}
1 & 1 & 2 & 2 & 1 & 1 & 1 & 1 \\
2 & 2 & 1 & 1 & 2 & 2 & 2 & 2 \\
3 & 3 & 3 & 3 & 3 & 3 & 3 & 3 \\
4 & 4 & 4 & 4 & 4 & 4 & 4 & 4 \\
\end{array}\right]\]
\textup{are both isomorphic to the trivial quandle of two elements, but
$CG(\{1,2\}\subset X)={1}$ while $CG(\{3,4\}\subset X)\cong\mathbb{Z}_2$.} 
\end{example}

\section{\Large \textbf{Enhancing the counting invariant}}\label{E}

We will now use the column group to define an enhancement of the blackboard 
birack counting invariants. 

\begin{definition}
\textup{Let $L=L_1\cup\dots\cup L_c$ be an oriented link of $c$ components
and $(X,B)$ a finite blackboard birack with birack rank $N$. The 
\textit{column group enhanced birack multiset invariant} is the multiset of 
column subgroups}
\[\phi_{(X,B)}^{CG,M}(L)=\left\{CG(\mathrm{Im}(f)\subset X)\ | \  
f\in \mathrm{Hom}(BBR(L,\mathbf{w}),(X,B)), \mathbf{w}\in(\mathbb{Z}_N)^c
\right\}\]
\textup{and the \textit{column group enhanced birack polynomial invariant} is}
\[\phi_{(X,B)}^{CG}(L)=\sum_{\mathbf{w}\in(\mathbb{Z}_N)^c)} 
\left(\sum_{f\in \mathrm{Hom}(BBR(L,\mathbf{w}),(X,B))} u^{|CG(\mathrm{Im}(f)\subset X)|}\right).\]
\end{definition}

That is, $\phi^{CG,M}_{(X,B)}(L)$ is the multiset of column subgroups of
the image subbiracks of labelings of a diagrams of $L$ by $(X,B)$ over a 
complete period of framings of $L$ modulo $N$. In the polynomial
version $\phi^{CG}_{(X,B)}(L)$ we trade some information (isomorphism
type of a column subgroup is replaced with its cardinality) to get a more
easily comparable invariant. In both cases, including column subgroup 
information enables the new invariants to distinguish between 
different labelings, resulting in a more sensitive invariant than simply 
counting labelings. Note that we can recover the integral birack counting 
invariant $\Phi^{\mathbb{Z}}_{(X,B)}(L)$ by specializing $u=1$ in
$\phi^{CG}_{(X,B)}(L)$ or by taking the cardinality of 
$\phi^{CG,M}_{(X,B)}(L)$.

\begin{example}
\textup{The trefoil knot $3_1$ has nine colorings by the $(t,s,r)$-birack
$X=\mathbb{Z}_3=\{1,2,3\}$ with $t=2, \ s=1, \ r=1$ -- in fact, these are 
the well-known Fox 3-colorings of the trefoil. Three of these labelings 
have singleton image subbiracks and six are surjective. Each element of $X$ 
has $\tau_i$ a transposition, so the column subgroups of the constant 
labelings are copies of $\mathbb{Z}_2$, while the surjective labelings 
have column subgroup generated by all three transpositions, i.e. isomorphic 
to $S_3$. Thus, the integral birack counting invariant value 
$\Phi_{(X,B)}^{\mathbb{Z}}(3_1)=|\mathrm{Hom}(BBR(3_1),(X,B))|=9$ with
column group enhancements becomes
$\phi^{CG,M}_{(X,B)}(L)=\{3\times \mathbb{Z}_2,6\times S_3\}$ or
$\phi^{CG}_{(X,B)}(L,T)(3_1)=3u^2+6u^6.$}
\begin{center}
{\begin{tabular}{|c|c|c|}
\hline
\includegraphics[width=1in]{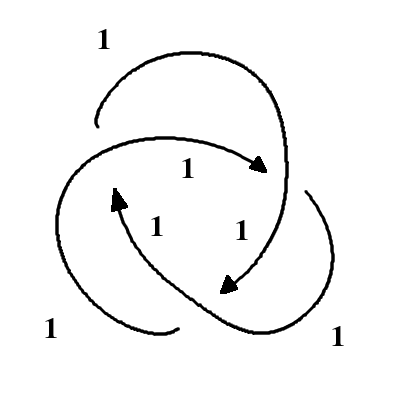}&%.75in]{trefoil11.png}&
\includegraphics[width=1in]{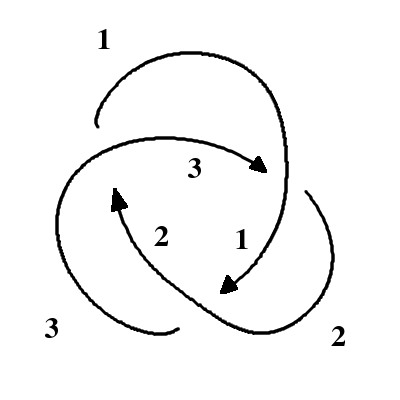}&%.75in]{trefoil12.png}&
\includegraphics[width=1in]{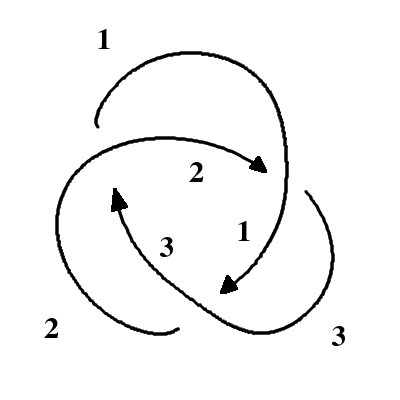}\\%.75in]{trefoil13.png}\\
\hline
\includegraphics[width=1in]{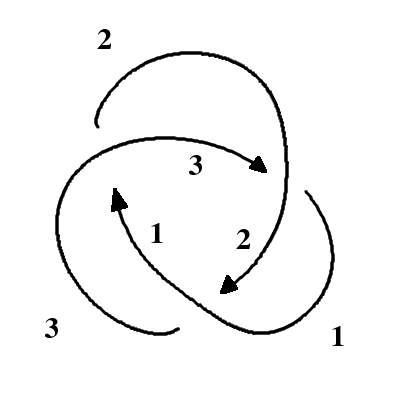}&%.75in]{trefoil21.png}&
\includegraphics[width=1in]{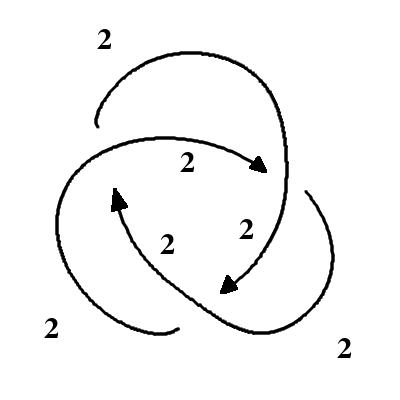}&%.75in]{trefoil22.png}&
\includegraphics[width=1in]{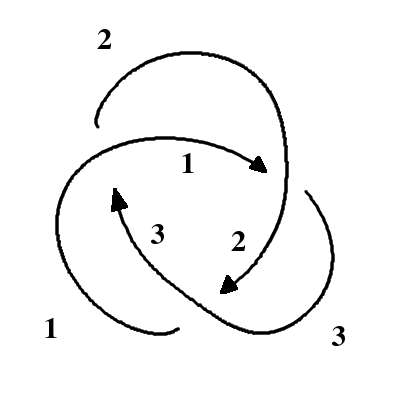}\\%.75in]{trefoil23.png}\\
\hline
\includegraphics[width=1in]{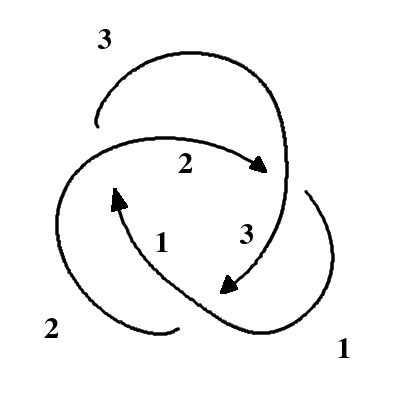}&%.75in]{trefoil31.png}&
\includegraphics[width=1in]{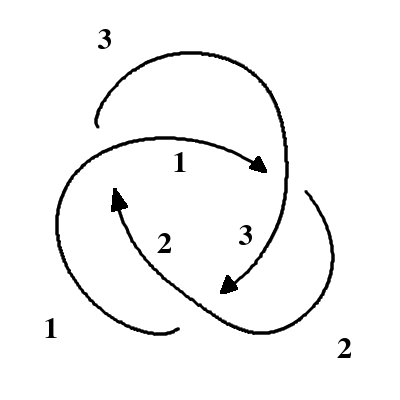}&%.75in]{trefoil32.png}&
\includegraphics[width=1in]{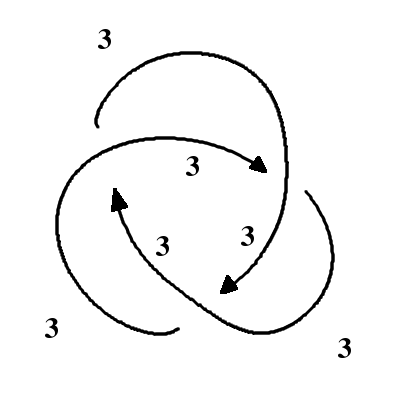}\\%.75in]{trefoil33.png}\\
\hline
\end{tabular}}
\end{center}
\end{example}

\begin{example}
\textup{Let us compute the column group enhanced birack counting invariant 
of the Hopf link with respect to the birack $X$ with birack matrix}
\[M_{(X,B)}=\left[\begin{array}{cccc|cccc} 
 2 & 2 & 2 & 2 & 1 & 1 & 1 & 1 \\
 1 & 1 & 1 & 1 & 2 & 2 & 2 & 2 \\
 3 & 3 & 3 & 3 & 3 & 3 & 4 & 4 \\
 4 & 4 & 4 & 4 & 4 & 4 & 3 & 3\\
\end{array}\right].\]
\textup{The birack rank of $(X,B)$ is 2, so we need to consider diagrams of the
Hopf link with both even and odd writhes on each component. The labeling rule
can be expressed as follows: semiarcs labeled $1$ switch to $2$ and $2$ 
switch to $1$ when crossing under any arc, semiarcs labeled $3$ or $4$ 
retain their label when crossing under any arc; semiarcs labeled $3$ switch 
to $4$ and $4$ switch to $3$ when crossing over a $3$ or $4$, and all other 
overcrossings retain their labels when crossing over. Note that the labelings 
of the diagrams
with writhe vectors $(0,1)$ and $(1,0)$ are the same due to the symmetry of 
the link. The reader can easily verify that the valid labelings are the ones
listed in the table.}
\[\begin{array}{c}
\begin{array}{|cc|cc|} \hline
\raisebox{-0.5in}{\includegraphics{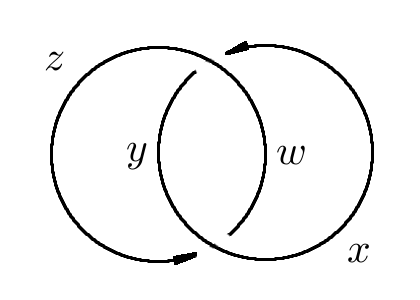}} & 
\begin{array}{cccc}
x & y  & z & w \\ \hline
- & -  & - & - \\
\end{array}
&
\raisebox{-0.5in}{\includegraphics{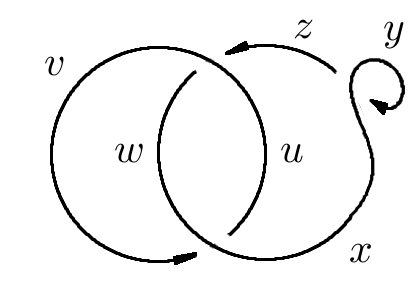}} &
\begin{array}{cccccc}
x & y & z & w & u & v \\ \hline
1 & 2 & 2 & 1 & 3 & 3 \\
1 & 2 & 2 & 1 & 4 & 4 \\
2 & 1 & 1 & 2 & 3 & 3 \\
2 & 1 & 1 & 2 & 4 & 4 \\
\end{array} \\ \hline \end{array} \\
\begin{array}{|ccc|} \hline
\raisebox{-0.5in}{\includegraphics{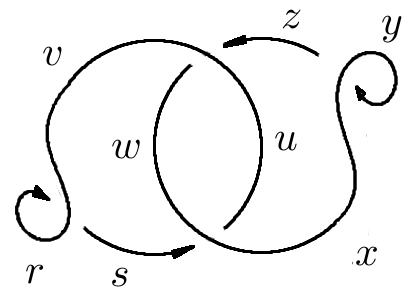}} & 
\begin{array}{cccccccc}
x & y & z & w & u & v & s & t \\ \hline
1 & 2 & 2 & 1 & 1 & 2 & 2 & 1 \\
1 & 2 & 2 & 1 & 2 & 1 & 1 & 2 \\
2 & 1 & 1 & 2 & 1 & 2 & 2 & 1 \\
2 & 1 & 1 & 2 & 2 & 1 & 1 & 2 \\
\end{array} &
\begin{array}{cccccccc}
x & y & z & w & u & v & s & t \\ \hline
3 & 3 & 4 & 4 & 3 & 3 & 4 & 4 \\
3 & 3 & 4 & 4 & 4 & 4 & 3 & 3 \\
4 & 4 & 3 & 3 & 3 & 3 & 4 & 4 \\
4 & 4 & 3 & 3 & 4 & 4 & 3 & 3 \\
\end{array} \\ \hline
\end{array}
\end{array}\]

\textup{
The integral birack counting invariant is thus $\Phi^{\mathbb{Z}}_{(X,B)}(L)=16$. The column group enhancement information distinguishes some of the labelings --
the image subbiracks of labelings include $\{1,2\}$, $\{3,4\}$ and $\{1,2,3,4\}$
with corresponding column groups $\mathbb{Z}_2$, $\mathbb{Z}_2$ and 
$\mathbb{Z}_2\oplus\mathbb{Z}_2$
respectively. Thus, the column group enhanced rack counting invariant is}
\[\phi^{CG}_{(X,B)}(L)=4u^2 + 12u^4\]
\textup{or in multiset form}
\[\phi^{CG,M}_{(X,B)}(L)=\{4\times \mathbb{Z}_2, 
12\times \mathbb{Z}_2\oplus\mathbb{Z}_2\}.\]

\end{example}

As we have seen, the integral counting invariant can be obtained as a 
specialization of the column group enhanced invariant.
Our last two examples show that the column group enhanced counting invariants
are strictly stronger than the unenhanced counting invariants.

\begin{example}
\textup{Consider the knots $5_1$ and $6_1$. Several 
other enhancements of counting invariants detect the difference between
these two knots despite having the same integral counting invariant value,
including generalized quandle polynomial enhancements and rack shadow
enhancements \cite{N2, CN2}. As expected, there is a blackboard birack 
$(X,B)$ whose column group enhancement distinguishes the knots $5_1$ and
$6_1$ while we have 
$\Phi^{\mathbb{Z}}_{(X,B)}(5_1)=30=\Phi^{\mathbb{Z}}_{(X,B)}(6_1)$.}
\[\scalebox{0.9}{$
M_{(X,B)}=
\left[\begin{array}{cccccccccc|cccccccccc}
1 & 3 & 5 & 2 & 4 & 2 & 1 & 5 & 4 & 3 & 
1 & 1 & 1 & 1 & 1 & 1 & 1 & 1 & 1 & 1 \\
5 & 2 & 4 & 1 & 3 & 4 & 3 & 2 & 1 & 5 & 
2 & 2 & 2 & 2 & 2 & 2 & 2 & 2 & 2 & 2 \\
4 & 1 & 3 & 5 & 2 & 1 & 5 & 4 & 3 & 2 &
3 & 3 & 3 & 3 & 3 & 3 & 3 & 3 & 3 & 3 \\
3 & 5 & 2 & 4 & 1 & 3 & 2 & 1 & 5 & 4 &
4 & 4 & 4 & 4 & 4 & 4 & 4 & 4 & 4 & 4 \\
2 & 4 & 1 & 3 & 5 & 5 & 4 & 3 & 2 & 1 &
5 & 5 & 5 & 5 & 5 & 5 & 5 & 5 & 5 & 5 \\
8 & 10 & 7 & 9 & 6 & 6 & 10 & 9 & 8 & 7 &
6 & 6 & 6 & 6 & 6 & 6 & 6 & 6 & 6 & 6 \\
7 & 9 & 6 & 8 & 10 & 8 & 7 & 6 & 10 & 9 &
7 & 7 & 7 & 7 & 7 & 7 & 7 & 7 & 7 & 7 \\
6 & 8 & 10 & 7 & 9 & 10 & 9 & 8 & 7 & 6 &
8 & 8 & 8 & 8 & 8 & 8 & 8 & 8 & 8 & 8 \\
10 & 7 & 9 & 6 & 8 & 7 & 6 & 10 & 9 & 8 &
9 & 9 & 9 & 9 & 9 & 9 & 9 & 9 & 9 & 9 \\
9 & 6 & 8 & 10 & 7 & 9 & 8 & 7 & 6 & 10 &
10 & 10 & 10 & 10 & 10 & 10 & 10 & 10 & 10 & 10 \\
\end{array}\right]$}
\]
\[
\begin{array}{cc}
\includegraphics{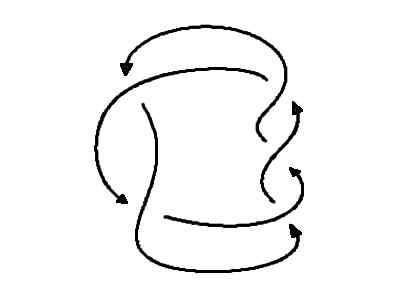} & \includegraphics{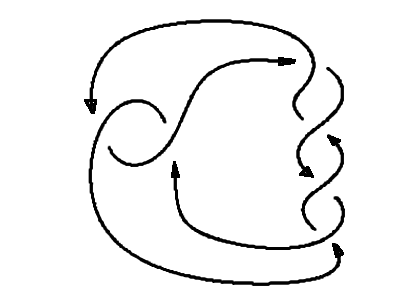} \\
\Phi^{CG}_{(X,B)}=5u^2+5u^4 + 20u^{10} & 
\Phi^{CG}_{(X,B)}=5u^2+5u^4 + 20u^{20} \\
\end{array}
\]
\end{example}

\begin{example}\label{ex3141}
\textup{Let  $(X,B)$ be the $27$-element conjugation quandle on the conjugation 
classes of $(13)(56)$ and $(15643)$ in $S_6$;, i.e. 
\[X=\{x\in S_6\ |\ x=y^{-1}(13)(56)y \ \mathrm{or}\ x=y^{-1}(15643)y\ 
\mathrm{for \ some}\ y\in S_6\}\] 
with birack operation \[B(x,y)=(x^{-1}yx,x).\]
Our \texttt{python} computations say that both the trefoil
$3_1$ and the figure eight $4_1$ have quandle counting invariant value
$|\mathrm{Hom}(3_1,T)|=|\mathrm{Hom}(4_1,T)|=147$; however, the column
group enhancement reveals distinct values. We list only the operation matrix 
for $B_1$ since $M_{B_2}$ is the trivial operation, i.e $B_2(x,y)=x$.}
\[
\begin{array}{rcl}
\phi^{CG}_{(X,B)}(3_1) & = & 12u^5+15u^2+60u^6+60u^{60} \\
\phi^{CG}_{(X,B)}(4_1) & = & 12u^5+15u^2 + 120u^{60}.
\end{array}
\]
\scalebox{0.7}{
$M_T=\left[\begin{array}{ccccccccccccccccccccccccccc}
1& 3& 4& 5& 6& 3& 5& 18& 23& 24& 16& 19& 4& 22& 23& 24& 22& 19& 18& 1& 1& 16& 15& 17& 17& 15& 6\\ 
7& 2& 8& 9& 10& 11& 12& 7& 7& 13& 2& 13& 14& 8& 14& 12& 27& 11& 13& 25& 11& 11& 11& 26& 8& 14& 12\\ 
5& 15& 3& 6& 4& 1& 16& 17& 18& 19& 1& 20& 6& 21& 20& 21& 18& 15& 22& 17& 16& 19& 3& 3& 22& 4& 5\\ 
6& 20& 1& 4& 3& 5& 19& 3& 22& 15& 17& 23& 5& 6& 19& 4& 23& 4& 21& 22& 24& 20& 17& 15& 24& 21& 1\\ 
3& 17& 6& 1& 5& 4& 23& 20& 6& 21& 24& 16& 3& 18& 5& 19& 21& 24& 23& 16& 17& 5& 20& 18& 1& 19& 4\\ 
4& 18& 5& 3& 1& 6& 20& 15& 21& 4& 22& 5& 1& 23& 16& 15& 6& 21& 6& 24& 22& 23& 18& 20& 16& 24& 3\\
2& 8& 14& 25& 12& 26& 7& 27& 10& 12& 12& 2& 8& 27& 25& 13& 8& 27& 25& 11& 25& 9& 10& 25& 7& 2& 10\\ 
12& 14& 2& 14& 26& 25& 2& 8& 27& 2& 7& 14& 9& 9& 13& 26& 7& 10& 26& 26& 9& 27& 26& 11& 27& 8& 7\\ 
26& 25& 12& 2& 25& 14& 14& 14& 9& 8& 27& 9& 11& 25& 12& 11& 12& 12& 10& 12& 8& 7& 13& 27& 11& 27& 8\\ 
14& 27& 25& 26& 2& 12& 27& 11& 26& 10& 26& 7& 7& 10& 11& 14& 14& 8& 9& 27& 13& 14& 7& 14& 12& 12& 11\\ 
25& 11& 26& 12& 14& 2& 9& 25& 25& 27& 11& 27& 10& 26& 10& 9& 13& 2& 27& 7& 2& 2& 2& 8& 26& 10& 9\\ 
8& 7& 9& 11& 7& 10& 10& 10& 12& 26& 13& 12& 2& 7& 9& 2& 9& 9& 14& 9& 26& 25& 27& 13& 2& 13& 26\\ 
27& 12& 27& 27& 27& 27& 26& 12& 2& 25& 14& 26& 13& 2& 8& 7& 11& 25& 2& 14& 10& 26& 9& 12& 14& 25& 13\\ 
10& 13& 7& 8& 11& 9& 13& 2& 8& 14& 8& 25& 25& 14& 2& 10& 10& 26& 12& 13& 27& 10& 25& 10& 9& 9& 2\\
23& 19& 17& 21& 15& 16& 24& 1& 17& 20& 3& 18& 23& 4& 15& 6& 20& 3& 24& 18& 19& 15& 1& 4& 21& 6& 16\\ 
17& 1& 21& 16& 23& 15& 6& 23& 5& 22& 19& 24& 15& 17& 6& 16& 24& 16& 20& 5& 3& 1& 19& 22& 3& 20& 21\\ 
16& 4& 15& 23& 21& 17& 1& 22& 20& 16& 5& 15& 21& 24& 18& 22& 17& 20& 17& 3& 5& 24& 4& 1& 18& 3& 23\\ 
19& 21& 20& 18& 24& 22& 17& 24& 15& 5& 6& 3& 22& 19& 17& 18& 3& 18& 1& 15& 23& 21& 6& 5& 23& 1& 20\\ 
18& 16& 22& 24& 20& 19& 21& 5& 1& 18& 15& 22& 20& 3& 4& 5& 19& 1& 19& 23& 15& 3& 16& 21& 4& 23& 24\\ 
20& 24& 18& 22& 19& 24& 22& 16& 3& 23& 4& 17& 18& 15& 3& 23& 15& 17& 16& 20& 20& 4& 5& 6& 6& 5& 19\\ 
21& 23& 16& 15& 17& 23& 15& 4& 24& 3& 18& 6& 16& 5& 24& 3& 5& 6& 4& 21& 21& 18& 22& 19& 19& 22& 17\\ 
24& 6& 19& 20& 22& 18& 3& 21& 19& 1& 23& 4& 24& 16& 22& 17& 1& 23& 3& 4& 6& 22& 21& 16& 20& 17& 18\\ 
15& 22& 23& 17& 16& 21& 18& 19& 4& 6& 21& 1& 17& 20& 1& 20& 4& 22& 5& 19& 18& 6& 23& 23& 5& 16& 15\\ 
22& 5& 24& 19& 18& 20& 4& 6& 16& 17& 20& 21& 19& 1& 21& 1& 16& 5& 15& 6& 4& 17& 24& 24& 15& 18& 22 \\ 
11& 26& 10& 7& 9& 8& 25& 13& 14& 9& 9& 11& 26& 13& 7& 27& 26& 13& 7& 2& 7& 12& 14& 7& 25& 11& 14 \\ 
9& 10& 11& 10& 8& 7& 11& 26& 13& 11& 25& 10& 12& 12& 27& 8& 25& 14& 8& 8& 12& 13& 8& 2& 13& 26& 25 \\ 
13& 9& 13& 13& 13& 13& 8& 9& 11& 7& 10& 8& 27& 11& 26& 25& 2& 7& 11& 10& 14& 8& 12& 9& 10& 7& 27 \\
\end{array}
\right]$}
\end{example}

\section{\large \textbf{Questions}}\label{Q}

In this section we collect questions for future research.

What kinds of groups can arise as column groups of a finite rack or quandle?
That is, given a finite group, can one construct a blackboard birack
with the specified column group? What is the relationship between the column 
group and birack polynomials?
 
A constant action rack always has a cyclic column group, generated by the
single column permutation appearing in the birack matrix; what 
can one say about the column groups of specific types of blackboard biracks
such as conjugation quandles, symplectic quandles, Coxeter racks, or
$(t,s,r)$-biracks?

In \cite{J}, a construction is given which expresses any quandle in terms
of a quandle structure on right cosets of the automorphism group of the
original quandle; is a similar construction possible starting with the
column group? What is the correct generalization of the column group to
infinite biracks?

To maximize sensitivity of the column group enhanced invariant,
we want blackboard biracks with as many subbiracks with distinct 
column subgroups as possible. On the other hand, biracks $(X,B)$ with 
larger cardinalities require more computation time. Finding fast 
algorithms for computing the sets of birack labelings of a diagram
will improve the practical utility of column group enhanced invariants.

Our \texttt{python}
code for computing the invariants defined in this paper is available 
at the second listed author's website, \texttt{www.esotericka.org}. 
Portions of this paper also appear in the first listed author's senior
thesis.

\bigskip

\noindent
\textsc{Department of Mathematics \\
University of California, San Diego\\
9500 Gilman Dr. \#0112\\
La Jolla, CA 92093-0112}

\noindent \textit{Email address:} \texttt{jhennig@math.ucsd.edu }

\bigskip

\noindent
\textsc{Department of Mathematical Sciences \\
Claremont McKenna College \\
850 Colubmia Ave. \\
Claremont, CA 91711}

\noindent
\textit{Email address:} \texttt{knots@esotericka.org}

\end{document}